\documentclass[11 pt]{amsart}
\usepackage{fancyhdr}
\usepackage[margin=1.15in]{geometry}
\usepackage{amsthm}
\usepackage{amsmath}
\usepackage{amssymb}
\usepackage{enumerate}
\usepackage{graphicx}
\usepackage{cancel}
\usepackage{hyperref}

\usepackage{comment}
\usepackage{enumitem}
\usepackage{mathtools}
\usepackage{tikz}
\usetikzlibrary{matrix,arrows,decorations.pathmorphing}
\usepackage{mathabx}
\usetikzlibrary{matrix}
\usepackage{scrextend}
\usepackage{tkz-graph}

\newcommand\blfootnote[1]{%
  \begingroup
  \renewcommand\thefootnote{}\footnote{#1}%
  \addtocounter{footnote}{-1}%
  \endgroup
}

\newtheorem{thm}{Theorem}[section]
\newtheorem{cor}[thm]{Corollary}

\DeclareMathOperator{\im}{im}
\DeclareMathOperator{\spa}{span}
\newcommand*{\HM}{\hspace{2pt}\widecheck{\hspace{-2pt}\textit{HM}}\hspace{0pt}}
\newcommand*{\HS}{\hspace{3pt}\widecheck{\hspace{-3pt}\textit{HS}\hspace{1pt}}\hspace{0pt}}

\theoremstyle{definition}
\newtheorem{lem}[thm]{Lemma}

\newtheorem{conj}[thm]{Conjecture}
\newtheorem{exmp}[thm]{Example}
\theoremstyle{remark}
\newtheorem{lem*}{Lemma}

 {\begin{list}{}%
         {\setlength{\leftmargin}{#1}}%
         \item[]%
 }
 {\end{list}}
\input xy
\xyoption{all}

\title{On the Pin(2)-Equivariant Monopole Floer Homology of Plumbed 3-Manifolds}
\author{Irving Dai}
\address{Princeton University, Princeton, NJ 08540, USA}
\email{idai@math.princeton.edu}

\begin{document}
\maketitle
\begin{abstract}  
We compute the Pin(2)-equivariant monopole Floer homology for the class of plumbed 3-manifolds considered by Ozsv\'ath and Szab\'o in \cite{OS}. We show that for these manifolds, the Pin(2)-equivariant monopole Floer homology can be calculated in terms of the Heegaard Floer/monopole Floer lattice complex defined by N\'emethi \cite{Nem3}. Moreover, we prove that in such cases the ranks of the usual monopole Floer homology groups suffice to determine both the Manolescu correction terms and the Pin(2)-homology as an abelian group. As an application of this, we show that $\beta(-Y, s) = \bar{\mu}(Y, s)$ for all plumbed 3-manifolds with at most one ``bad" vertex, proving (an analogue of) a conjecture posed by Manolescu in \cite{Man3}. Our proof also generalizes results by Stipsicz \cite{Stip} and Ue \cite{Ue} relating $\bar{\mu}$ with the Ozsv\'ath-Szab\'o $d$-invariant. Some observations aimed at extending our computations to manifolds with more than one bad vertex are included at the end of the paper.
\end{abstract}
        
\blfootnote{This work was carried out with the support of NSF grant number DGE 1148900.} 
\begin{section}{Introduction}\label{sec1}
\noindent
The goal of this paper is to compute the Pin(2)-equivariant monopole Floer homology of a certain family of plumbed 3-manifolds using the lattice cohomology construction of N\'emethi \cite{Nem3}. First introduced by Manolescu \cite{Man1} and further developed by Lin \cite{Lin1}, Pin(2)-equivariant monopole Floer homology is a modification of the usual Seiberg-Witten Floer homology for 3-manifolds that takes advantage of an extra $\mathbb{Z}/2\mathbb{Z}$-symmetry in the Chern-Simons-Dirac functional. Our approach in this paper is to use the lattice cohomology framework for computing Heegaard Floer homology developed in e.g.\ \cite{OS}, \cite{Nem1}, \cite{Nem3}, together with a Gysin sequence relating the usual and Pin(2)-equivariant monopole Floer homologies. We show that in the case of plumbed 3-manifolds with at most one ``bad" vertex \cite{OS} or, more generally, almost-rational plumbings \cite{Nem3}, the Pin(2)-homology is in fact determined by the lattice complex. This class of 3-manifolds does not include all plumbed 3-manifolds, but is large enough to contain (for example) all Seifert fibered rational homology spheres. \\
\\
Our approach to computing the Pin(2)-homology by constraining it through the Gysin sequence is taken from \cite{Lin2}, in which it is used to compute the Pin(2)-homology of various Seifert spaces. After completing the current work, the author also learned that many of the same results for Seifert rational homology spheres have been obtained by Stoffregen in \cite{Stoff}. The approach there involves explicitly understanding Manolescu's Seiberg-Witten Floer spectrum (see e.g.\ \cite{Man2}), whereas our approach is more combinatorial in nature. Theorems \ref{prop1} and \ref{prop2} below should be compared with (for example) Theorem 1.1 and Corollary 1.2 in Stoffregen's work. \\
\\
It should be noted that we work with the formulation of Pin(2)-equivariant monopole Floer homology developed by Lin, rather than the original definition via Seiberg-Witten Floer spectra given by Manolescu. Currently, these two theories are only conjecturally isomorphic (see \cite{Lid} for the non-equivariant case). Thus, our results are technically only valid in Lin's setting, whereas (for example) Stoffregen's computations hold for Manolescu's original formulation. However, we expect that the algebraic arguments given in this paper can be carried out for both theories. \\
\\
The organization of this paper is as follows. In the next two subsections, we review the details of Pin(2)-equivariant monopole Floer homology and lattice cohomology that we will need for our computations. In Section \ref{sec2}, we prove the following theorems on the Pin(2)-homology of plumbed 3-manifolds, outlined here for motivation: 

\begin{thm}\label{prop1}
Let $Y$ be a rational homology 3-sphere given by surgery on a connected, negative-definite graph with at most one bad vertex (in the sense of \cite{OS}). Let $s$ be a self-conjugate spin$^c$ structure on $Y$. Then the orientation-reversed Pin(2)-equivariant monopole Floer homology $\HS(-Y, s)$ may be computed from the lattice complex of Nem\'ethi \cite{Nem3}.
\end{thm}

\begin{thm}\label{prop2}
Let $Y$ be a rational homology 3-sphere given by surgery on a connected, negative-definite graph with at most one bad vertex (in the sense of \cite{OS}). Let $s$ be a self-conjugate spin$^c$ structure on $Y$. Then the ranks of the usual monopole Floer homology determine both the Manolescu correction terms of $\HS(-Y, s)$ and also $\HS(-Y, s)$ as an abelian group.
\end{thm}

\noindent
The family of 3-manifolds described above may be enlarged to the class of all 3-manifolds obtained by plumbings on almost-rational graphs (see \cite{Nem3}) with no additional difficulty, but we work in the original setting of \cite{OS} for convenience of exposition. For the precise statements of these theorems, see Theorems \ref{prop6} and \ref{prop7} below. In Section \ref{sec3}, we relate the Neumann-Siebenmann invariant $\bar{\mu}(Y, s)$ (defined in \cite{Neu}, \cite{Sieb}) with the lattice cohomology of $(Y, s)$. Combined with Theorem \ref{prop7}, this allows us to prove:

\begin{thm}\label{propconj}
Let $Y$ be a rational homology 3-sphere given by surgery on a connected, negative-definite graph with at most one bad vertex (in the sense of \cite{OS}). Let $s$ be a spin structure on $Y$ (which we may view as a self-conjugate spin$^c$ structure). Then $\beta(-Y, s) = \bar{\mu}(Y, s)$.
\end{thm}
\noindent
For Seifert integer homology spheres, this result is again due to Stoffregen \cite{Stoff} using a computation of $\bar{\mu}$ for such manifolds by Ruberman and Saveliev \cite{Ruber}; in our case, it is a straightforward corollary of the lattice cohomology framework. In particular, Theorem \ref{propconj} proves (in greater generality) Conjecture 4.1 of \cite{Man3}, albeit for the formulation of Pin(2)-equivariant homology given by Lin. As a byproduct of the proof, we also obtain a Heegaard Floer/monopole Floer homology characterization of $\bar{\mu}$ for all plumbed 3-manifolds with at most one bad vertex. This generalizes work of Stipsicz \cite{Stip} and Ue \cite{Ue} relating $\bar{\mu}(Y, s)$ and the Ozsv\'ath-Szab\'o $d$-invariant for rational surface singularities and spherical 3-manifolds, respectively. \\
\\
Finally, we end with some observations and examples aimed at extending our computations to manifolds with more than one bad vertex. Throughout the paper, we always work over the field of two elements $\mathbb{F} = \mathbb{F}_2$. 

\begin{subsection}{Pin(2)-Equivariant Monopole Floer Homology}
We begin by reviewing the essential tenets of Pin(2)-equivariant monopole Floer homology as given in \cite{Lin1}. Recall that for a closed 3-manifold $Y$ equipped with a spin$^c$ structure $s$, the monopole Floer homology groups are defined by studying the chain complex generated by critical points of the (perturbed) Chern-Simons-Dirac functional on a certain ``configuration space" associated to $Y$ and $s$. (See e.g.\ \cite{KM} for details.) This is a gauge-theoretic invariant which assigns to $(Y, s)$ three groups fitting into the long exact sequence
\[
\cdots \rightarrow \widebar{\textit{HM}\hspace{3pt}}\hspace{-3pt}(Y, s) \rightarrow \HM(Y, s) \rightarrow \hspace{2pt}\widehat{\hspace{-2pt}\textit{HM}}(Y, s) \rightarrow \cdots.
\]
Each of these groups are modules over the ring $\mathbb{F}[U]$, which may be thought of as the $S^1$-equivariant cohomology of a point. The monopole Floer homology groups have an absolute $\mathbb{Z}/2\mathbb{Z}$-grading and also a more refined relative $\mathbb{Z}/d\mathbb{Z}$-grading, where $d$ depends on the choice of $s$. If $s$ has torsion Chern class, then the latter is a relative $\mathbb{Z}$-grading. The action of $U$ has degree $-2$. \\
\\
If $Y$ is a rational homology sphere, then the relative $\mathbb{Z}$-grading corresponding to any spin$^c$ structure becomes an absolute $\mathbb{Q}$-grading, and we define the \textit{Fr\o yshov invariant} of $(Y, s)$ as follows \cite{Froy}. Let $\mathcal{U}_d^+$ be the $\mathbb{F}[U]$-module $\mathbb{F}[U^{-1}, U]/U\mathbb{F}[U]$, shifted so that the element $1$ has grading $d$. The monopole Floer homology $\HM(Y, s)$ decomposes into the direct sum of a finite part and a single infinite tower $\mathcal{U}_d^+$, the latter of which is canonically determined as the image of $\widebar{\textit{HM}\hspace{3pt}}\hspace{-3pt}(Y, s)$ in the long exact sequence above. The Fr\o yshov invariant $\delta(Y, s) = d/2 \in \mathbb{Q}$ is defined to be half of the grading shift of this infinite $U$-tower. \\
\\
In the case that $s$ is self-conjugate, it turns out that the Chern-Simons-Dirac functional has an extra $\mathbb{Z}/2\mathbb{Z}$-symmetry, which allows us to consider the subcomplex consisting of $\mathbb{Z}/2\mathbb{Z}$-invariant chains of critical points and flows between such chains. The main analytical difficulty in doing this is that for perturbations of the Chern-Simons-Dirac functional which preserve the $\mathbb{Z}/2\mathbb{Z}$-symmetry, critical points necessarily occur in entire submanifolds, rather than isolated points. Nevertheless, once the appropriate theory is defined (see \cite{Lin1}), taking the homology of this invariant subcomplex yields a similar triple of gauge-theoretic invariants fitting into the analogous long exact sequence
\[
\cdots \rightarrow \widebar{\textit{HS}\hspace{3pt}}\hspace{-2pt}(Y, s) \rightarrow \HS(Y, s) \rightarrow \hspace{3pt}\widehat{\hspace{-3pt}\textit{HS}\hspace{1pt}}\hspace{0pt}(Y, s) \rightarrow \cdots.
\]
Each of these groups are modules over the ring 
\[
\mathcal{R} = \mathbb{F}[V][Q]/(Q^3),
\]
which may be thought of as the Pin(2)-equivariant cohomology of a point. For each self-conjugate spin$^c$ structure $s$, the Pin(2)-equivariant monopole Floer homology groups are graded by the same object as their non-equivariant counterparts, and in particular have a relative $\mathbb{Z}$-grading. The actions of $V$ and $Q$ have degrees $-4$ and $-1$, respectively. \\
\\
The relation between the usual and Pin(2)-equivariant monopole Floer homologies (for a self-conjugate spin$^c$ structure $s$ on $Y$) is expressed by the Gysin sequence (see Section 4.3 of \cite{Lin1})
\begin{equation}\label{eq1}
\cdots \xrightarrow{\cdot Q} \HS(Y, s) \xrightarrow{\iota_*} \HM(Y, s) \xrightarrow{\pi_*} \HS(Y, s) \xrightarrow{\cdot Q} \HS(Y, s) \xrightarrow{\iota_*} \cdots,
\end{equation}
which should be thought of as analogous to the usual Gysin sequence in algebraic topology for $S^0$-bundles. (There are of course similar sequences for the other two flavors of monopole Floer homology.) This is a map of graded $\mathcal{R}$-modules, where $V$ acts on the usual monopole Floer homology as $U^2$ and $Q$ acts as zero. The maps $\iota_*$ and $\pi_*$ in the sequence preserve the grading, while multiplication by $Q$ has grading $-1$. \\
\\
If $Y$ is a rational homology sphere, then (as in the non-equivariant case) the relative $\mathbb{Z}$-grading on the Pin(2)-monopole Floer homology becomes an absolute $\mathbb{Q}$-grading, and we may define the three \textit{Manolescu correction terms} as follows \cite{Man3}. Let $\mathcal{V}_d^+$ be the $\mathbb{F}[V]$-module $\mathbb{F}[V^{-1}, V]/V\mathbb{F}[V]$, shifted so that the element $1$ has grading $d$. The Pin(2)-equivariant monopole Floer homology $\HS(Y, s)$ decomposes into the direct sum of a finite part and the sum of three infinite towers $\mathcal{V}_c^+ \oplus \mathcal{V}_b^+ \oplus \mathcal{V}_a^+$, where the action of $Q$ sends the $c$-tower onto the $b$-tower and the $b$-tower onto the $a$-tower. (Again, the three towers are canonically determined, even though the finite part of the decomposition is not.) We define $\alpha(Y, s) \geq \beta(Y, s) \geq \gamma(Y, s)$ to be the rational numbers such that
\[
a = 2\alpha(Y, s), b = 2\beta(Y, s) + 1, \text{and } c = 2\gamma(Y, s) + 2.
\]
See \cite{Man2} and \cite{Man3} for Manolescu's disproof of the triangulation conjecture using these invariants.

\end{subsection}

\begin{subsection}{Lattice Cohomology}
We now review the construction of lattice cohomology as given in \cite{Nem3}. This formulation is slightly removed from the original setup of e.g.\ \cite{OS}, \cite{Nem1}, but has the advantage of being somewhat conceptually and combinatorially clearer. For any positive integer $s$, let $\mathbb{Z}^s$ be the $s$-dimensional integer lattice, and denote by $\mathcal{Q}_q$ the set of $q$-dimensional side-length-one lattice cubes in $\mathbb{Z}^s$. (Thus $\mathcal{Q}_0$ is the set of lattice points, $\mathcal{Q}_1$ is the set of lattice edges, and so on.) Let $\{w_q\}$ (often denoted simply by $w$) be a collection of functions $w_q: \mathcal{Q}_q \rightarrow \mathbb{Z}$ from the set of all such cubes into $\mathbb{Z}$. We refer to such a set $\{w_q\}$ as a \textit{collection of weight functions}. Define $S_n \subseteq \mathbb{Z}^s$ to be the sublevel set of cubes (of any dimension) whose weights are at most $n$; i.e.,
\[
S_n(\mathbb{Z}^s, w) = \bigcup_q \{\Box_q \in \mathcal{Q}_q : w_q(\Box_q) \leq n \}.
\]
For each non-negative integer $i$, we then define
\[
\mathbb{S}^i(\mathbb{Z}^s, w) = \bigoplus_n H^i(S_n, \mathbb{F}),
\]
where $H^i(S_n, \mathbb{F})$ is the usual $i$-dimensional $\mathbb{Z}/2\mathbb{Z}$-cohomology of the sublevel set $S_n$, and the sum is taken over all $n$. We give this object a grading by declaring an element of $H^i(S_n, \mathbb{F})$ to have grading $2n$, and we define a $U$-action on $\mathbb{S}^i(\mathbb{Z}^s, w)$ by setting multiplication by $U$ on $H^i(S_n, \mathbb{F})$ to be equal to the map on cohomology
\[
i^*: H^i(S_{n}, \mathbb{F}) \rightarrow H^i(S_{n-1}, \mathbb{F})
\]
induced by the inclusion of $S_{n-1}$ into $S_n$. We then define the lattice cohmology of $(\mathbb{Z}^s, w)$ by putting all these objects together and letting $i$ vary: 
\[
\mathbb{H}^*(\mathbb{Z}^s, w) = \mathbb{S}^*(\mathbb{Z}^s, w).
\]
\\
Now let $Y$ be a rational homology sphere with a fixed spin$^c$ structure $s$. Let $\Gamma$ be a plumbing diagram for $Y$, so that $Y$ is the boundary of the 4-manifold $W(\Gamma)$ constructed by attaching 2-handles to $B^4$ according to the decorated graph $\Gamma$. We have a preferred basis of $H_2(W(\Gamma), \mathbb{Z})$ formed by capping off the cores of the attached 2-handles inside $B^4$, which gives an isomorphism between $H_2(W(\Gamma), \mathbb{Z})$ and the integer lattice $L_\Gamma$ spanned by these basis elements. Similarly, taking the co-cores of the 2-handles provides a preferred basis of $H_2(W(\Gamma), Y, \mathbb{Z})$ and identifies the relative homology with another integer lattice $L'_\Gamma$. The homology exact sequence 
\[
0 \rightarrow H_2(W(\Gamma), \mathbb{Z}) \rightarrow H_2(W(\Gamma), Y, \mathbb{Z}) \rightarrow H_1(Y, \mathbb{Z}) \rightarrow 0
\]
identifies $L_\Gamma$ as a sublattice of $L'_\Gamma$, and the intersection pairing on the former extends to a $\mathbb{Q}$-valued intersection pairing on the latter. We define the set of \textit{characteristic vectors} on $\Gamma$ to be the set 
\[
\kappa = \{k \in L'_\Gamma : (k, x) = (x, x) \text{ mod 2 for all } x \in L_\Gamma\}. 
\]
There is a natural action of $L_\Gamma$ on $\kappa$ given by $k \mapsto k + 2L_\Gamma$, and we denote the orbit of a characteristic vector $k$ under this action by $[k]$. Noting that each element of $\kappa$ corresponds to a spin$^c$ structure on $W(\Gamma)$, it is easily seen that $[k]$ consists of the set of spin$^c$ structures on $W(\Gamma)$ limiting to a particular spin$^c$ structure on $Y$. We may thus identify the set of orbits $\{[k]\}$ with the set of spin$^c$ structures on $Y$; under this identification, the self-conjugate spin$^c$ structures on $Y$ correspond to those orbits that lie in the sublattice $L_\Gamma$. \\
\\
Now suppose that we have fixed a characteristic vector $k$ whose spin$^c$ structure limits to $s$ on $Y$. We put a weight function $w$ on $L_\Gamma$ as follows:
\begin{enumerate}
\item For each lattice point $x \in L_\Gamma$, we set $w_0(x) = -((x, x) + (x, k))/2$; and
\item For each $q$-dimensional cube $\Box_q$, we let $w_q(\Box_q)$ be the maximum of $w_0(x)$ for $x$ ranging over the vertices of $\Box_q$.
\end{enumerate}
We denote the lattice cohomology of $\Gamma$ with respect to this weight function by $\mathbb{H}^*(\Gamma, k) = \mathbb{H}^*(L_\Gamma, w)$. The key result is that if the graph $\Gamma$ is sufficiently ``nice", then the resulting object is the same (up to a grading shift) as the Heegaard Floer/monopole Floer homology of $(-Y, s)$. (Here, we use the isomorphism of Kutluhan-Lee-Taubes \cite{KLT}, Colin-Ghiggini-Honda \cite{CGH}, and Taubes \cite{Taub} and make little distinction between the Heegaard Floer and monopole Floer homology throughout.) More precisely, we say that a vertex $v$ of $\Gamma$ is ``bad" in the sense of \cite{OS} if the inequality $m(v) > -d(v)$ holds, where $m(v)$ is the decoration of $\Gamma$ at $v$ and $d(v)$ is the valency of $v$. Then:

\begin{thm}[Theorem 1.2 of \cite{OS}, Theorem 5.2.2 of \cite{Nem3}]\label{prop3}
Let $Y$ be a rational homology 3-sphere obtained by surgery on a connected, negative-definite graph $\Gamma$ with at most one bad vertex. Let $s$ be a spin$^c$ structure on $Y$, and let $k$ be any characteristic vector on $\Gamma$ whose corresponding spin$^c$ structure on $W(\Gamma)$ limits to $s$. Let $\sigma$ be the rational grading shift
\[
\sigma = \sigma(\Gamma, k) = - \dfrac{1}{4}(|\Gamma| + k^2),
\]
where $|\Gamma|$ is the number of vertices in $\Gamma$. Then the following are true:
\begin{enumerate}
\item $\mathbb{H}^q(\Gamma, k) = 0$ for all $q > 0$,
\item $\HM_{\text{even}}(-Y, s) \cong \mathbb{H}^*(\Gamma, k)[\sigma] = \mathbb{H}^0(\Gamma, k)[\sigma]$ as graded $\mathbb{F}[U]$-modules, and
\item $\HM_{\text{odd}}(-Y, s) = 0$.
\end{enumerate}
Here, $\mathbb{H}^0(\Gamma, k)[\sigma]$ is the lattice cohomology of $(\Gamma, k)$ shifted by grading $\sigma$, so that an element which had grading zero before now has grading $\sigma$.
\end{thm}
\noindent
(Theorem \ref{prop3} as stated in \cite{OS} and \cite{Nem3} deals with Heegaard Floer homology; see \cite{KMOS} for results concerning the monopole Floer homology directly.) \\
\\
Again, this is not largest possible class of manifolds for which the isomorphism is true; for a more general class of manifolds, see \cite{Nem3}. We will often keep the isomorphism of Theorem \ref{prop3} implicit and abuse notation by referring to elements of the lattice cohomology as lying in the monopole Floer homology, and vice-versa. In these cases, we will sometimes also suppress the grading shift by $\sigma$ and describe elements of the monopole Floer homology as having their lattice cohomology grading. Note that when computing the lattice cohomology, any representative of $[k]$ may be chosen; it is easily seen that choosing a different representative has the effect of shifting the grading, which is cancelled out by the change in $\sigma(\Gamma, k)$.
\end{subsection}
\end{section}

\begin{section}*{Acknowledgements}
\noindent
The author would like to thank his advisor, Zolt\'an Szab\'o, for suggesting the current project and for his continued support and guidance. The author would also like to thank the referee for several helpful comments and suggestions. This work was carried out with the support of NSF grant number DGE 1148900.
\end{section}

\begin{section}{Pin(2)-Equivariant Monopole Homology of Plumbed 3-Manifolds}\label{sec2}
\noindent
We now turn to the computation of Pin(2)-equivariant monopole homology for the class of plumbed 3-manifolds considered in \cite{OS}. The key step will be to observe that the composition of maps $\iota_* \circ \pi_*$ in the Gysin sequence (\ref{eq1}) can be identified with an obvious $\mathbb{Z}/2\mathbb{Z}$-symmetry in the lattice cohomology $\mathbb{H}^0(\Gamma, k)$. It turns out that understanding this involution provides a sufficient algebraic constraint in the Gysin sequence to completely determine the Pin(2)-homology. We thus begin by defining this extra symmetry and establishing some of its properties. Throughout, we assume that we are in the situation of Theorem \ref{prop3}, so that $Y$ is given by plumbing on a connected, negative-definite graph $\Gamma$ with at most one bad vertex. In addition, let $s$ be a self-conjugate spin$^c$ structure on $Y$ and let $k$ be a characteristic vector on $\Gamma$ whose corresponding spin$^c$ structure on $W(\Gamma)$ limits to $s$.

\begin{subsection}{The $J$ structure of Lattice Cohomology}
Let $J$ be the map on $L_\Gamma$ given by reflection through the point $-k/2$; i.e.,
\[
Jx = -x - k.
\] 
Note that since $s$ is self-conjugate, this indeed takes $L_\Gamma$ to itself, and it is straightforward to check that the weight function $w$ is invariant under the action of $J$. Hence $J$ preserves each sublevel set $S_n$ and induces a map on each cohomology group $H^i(S_n, \mathbb{F})$. This defines a $U$-equivariant involution on the entire lattice cohomology, which we also denote by $J$. Observing that $1 + J$  squares to zero, we can then take the homology of $\mathbb{H}^0(\Gamma, k)$ with respect to $1 + J$. For reasons that will become clear in Section \ref{sec3}, we refer to this homology as the \textit{derived lattice cohomology}, and denote it by $\mathbb{H}'(\Gamma, k)$. \\
\\
The derived lattice cohomology is easily described explicitly in terms of a particular basis of $\mathbb{H}^0(\Gamma, k)$, as follows. Observe that the connected components of each sublevel set $S_n$ provide a preferred basis for $\mathbb{H}^0(\Gamma, k)$, the members of which may be further subdivided into two types. First, there are the connected components of $S_n$ which are taken to themselves under the action of $J$; we label these basis elements by $F_i$ and denote their span by $F$. Second, there are the connected components that occur in pairs $E_i$ and $JE_i$; we denote the span of $\{E_i\}$ and $\{JE_i\}$ by $E$. We thus obtain the decomposition
\[
\mathbb{H}^0(\Gamma, k) = E \oplus F = \spa \{E_i\} \oplus \spa \{(1+J)E_i\} \oplus F.
\]
Note that the above splitting does not respect the $U$-action - while it is easily seen that $U$ maps the span of $\{(1+J) E_i\}$ into itself, the image of $F$ under $U$ in general lies in the subspace $\spa \{(1+J) E_i\} \oplus F$. Observing that $\ker (1 + J) = \spa \{(1+J) E_i\} \oplus F$ and $\im (1 + J) = \spa \{(1+J) E_i\}$, we evidently have the isomorphism of graded $\mathbb{F}[U]$-modules
\[
\mathbb{H}'(\Gamma, k) \cong (\spa \{(1+J)E_i\} \oplus F) / \spa \{(1+J)E_i\},
\] 
where we have written the right-hand side as a quotient to emphasize the $U$-module structure. As an abelian group, of course, $\mathbb{H}'(\Gamma, k)$ may be identified with $F$. \\
\\
We now prove a preliminary result on the structure of the derived lattice cohomology. 

\begin{lem}\label{prop4}
Let $(Y, s)$ and $(\Gamma, k)$ be as above. Then the derived lattice cohomology $\mathbb{H}'(\Gamma, k)$ is isomorphic as an $\mathbb{F}[U]$-module to $\mathcal{U}_r^+$ for some $r$.
\end{lem}

\begin{proof}
We begin by showing that in each grading, $F$ is at most one-dimensional. Let $X$ be a $J$-invariant connected component of a sublevel set $S_n$. By Theorem \ref{prop3}, all the higher cohomology groups of $X$ vanish. Hence the Lefschetz fixed-point theorem tells us that the action of $J$ on $X$ has at least one fixed point. (Note that in our setup, we are working over $\mathbb{Z}/2\mathbb{Z}$-coefficients, but Theorem \ref{prop3} actually holds for $\mathbb{Z}$-coefficients.) On the other hand, $J$ is geometrically given by reflection through $-k/2$, so the only possible fixed point is $-k/2$. Thus $X$ is uniquely specified as connected component of $S_n$ by the condition that it contains $-k/2$, showing that $F$ is at most one-dimensional in each grading $2n$. \\
\\
We now observe that if $F$ is zero in a particular grading, then it must be zero in all lower gradings. Indeed, if there are no $J$-invariant connected components in grading $2n$, then there cannot be any $J$-invariant connected components in gradings $2m \leq 2n$, since the sublevel sets $S_m$ are subsets of the sublevel set $S_n$. Moreover, it is clear from the geometric picture that if $F$ is nonzero in gradings $2n$ and $2n-2$, then the action of $U$ must map the nonzero element of $F$ in grading $2n$ to the nonzero element of $F$ in grading $2n-2$, plus possibly some elements of $\spa \{(1 + J)E_i\}$. This establishes the $U$-module structure and proves the claim.
\end{proof}
\noindent
Note that the fact that $F$ is zero- or one-dimensional implies the following structure result for the monopole Floer homology:
\begin{cor}\label{prop5}
Let $(Y, s)$ and $(\Gamma, k)$ be as above. Then there exists a unique grading $\rho = \rho(Y, s)$ such that in gradings $\rho + 2n$ for $n \geq 0$, the monopole Floer homology of $(-Y, s)$ has odd rank, and in gradings $\rho + 2n$ for $n < 0$, the monopole Floer homology of $(-Y, s)$ has even rank.
\end{cor}
\begin{proof}
The corollary immediately follows from the fact that $E$ is even-dimensional in each grading. The rational number $\rho$ is of course given by $r + \sigma$, where $r$ is as in Lemma \ref{prop4} and $\sigma$ is as in Theorem \ref{prop3}.
\end{proof}
\noindent
In Section $\ref{sec3}$, we will in fact show that this ``parity invariant" $\rho = \rho(Y, s)$ is equal to twice the Neumann-Siebenmann invariant $\bar{\mu}(Y, s)$. At the moment, however, we will restrict our attention to its relation with the Pin(2)-equivariant monopole Floer homology of $(-Y, s)$. \\
\\
We are now in a position to state our main theorem(s). In order to describe the Pin(2)-homology as an $\mathcal{R}$-module, it will be convenient for us to break it up into four $\mathbb{F}[V]$-submodules. To this end, we establish the following notation. Let $A$ and $B$ be two $\mathbb{Z}$-graded $\mathbb{F}[V]$-modules with graded parts $A_q$ and $B_q$, and let $n \in \mathbb{Z}/4\mathbb{Z}$. We write $A =_{[n]} B$ if the submodules 
\[
A_{[n]} = \bigoplus_{q = n \text{ mod } 4}A_q
\text{ \ \ \ \ \ and \ \ \ \ \ }
B_{[n]} = \bigoplus_{q = n \text{ mod } 4}B_q
\]
consisting of the sums of the groups in gradings congruent to $n$ modulo $4$ are isomorphic as graded $\mathbb{F}[V]$-modules. We refer to the above $\mathbb{F}[V]$-submodules as the $[n]$-\textit{submodules} of $A$ and $B$. Note that if $A$ is in fact a $\mathcal{R}$-module, then in addition we have an action of $Q$ which maps $A_{[n]}$ into $A_{[n-1]}$ for each $n$. \\
\\
We now make a precise re-formulation of Theorem \ref{prop1}. The explicit statement is rather cumbersome, but the rationale behind the casework will become clear after we embark upon the proof, which is given in the next subsection. 

\begin{thm}\label{prop6}
Let $s$ be a spin$^c$ structure on $Y$, and let $k$ be any characteristic vector on $\Gamma$ whose corresponding spin$^c$ structure on $W(\Gamma)$ limits to $s$. Let $r$ be as in Lemma \ref{prop4} and $\sigma$ be as in Theorem \ref{prop3}. Then we have the following set of $\mathbb{F}[V]$-module isomorphisms:
\begin{align*}
\HS(-Y, s)[-\sigma] &=_{[r+3]} 0, \\
\HS(-Y, s)[-\sigma] &=_{[r+2]} \ker(1+J), \\
\HS(-Y, s)[-\sigma] &=_{[r+1]} \mathbb{H}'(\Gamma, k)[1], \text{ and }\\
\HS(-Y, s)[-\sigma] &=_{[r]} \mathbb{H}^0(\Gamma, k)/\im (1+J),
\end{align*}
where $V$ acts on all groups on the right-hand side by $U^2$. The action of $Q$ from the $[r + 2]$- to the $[r + 1]$-submodule may be identified with the quotient map 
\[
\ker(1 + J) \rightarrow \mathbb{H}'(\Gamma, k) = \ker(1 + J)/\im(1 + J),
\] 
followed with multiplication by $U$ in $\mathbb{H}'(\Gamma, k)$. The action of $Q$ from the $[r+1]$- to the $[r]$-submodule may be identified with the inclusion map
\[
\mathbb{H}'(\Gamma, k) = \ker(1 + J)/\im(1 + J) \rightarrow \mathbb{H}^0(\Gamma, k)/\im (1+J).
\]
Multiplication by $Q$ is zero in all other cases.
\end{thm}
\noindent
\\
We now use Theorem \ref{prop6} to provide a precise re-formulation of Theorem \ref{prop2}. Note that in the case where $Y$ is expressed as surgery on a graph $\Gamma$ for which Theorem \ref{prop3} applies, the Fr\o yshov invariant $\delta(-Y, s)$ may also be defined as half of the grading below which the monopole Floer homology itself vanishes. This follows from the fact that the action of $U$ on the lattice cohomology is nonzero in each grading in which it is possible to be nonzero. \\

\begin{thm}\label{prop7}
Let $s$ be a spin$^c$ structure on $Y$, and let $k$ be any characteristic vector on $\Gamma$ whose corresponding spin$^c$ structure on $W(\Gamma)$ limits to $s$. Then the Manolescu correction terms for $(-Y, s)$ are given by: 
\begin{align*}
a &= 2\alpha(-Y, s) = \rho,\\
b &= 2\beta(-Y, s) + 1 = \rho + 1, \text{ and }\\
c &= 2\gamma(-Y, s) + 2 = 
\begin{cases}
2\delta(-Y, s) + 2 &\quad\text{if } 2\delta(-Y, s) = \rho \mod 4 \\
2\delta(-Y, s) &\quad\text{if } 2\delta(-Y, s) = \rho + 2 \mod 4. \ 
\end{cases}
\end{align*}
Moreover, the ranks of the Pin(2)-homology are as follows. In odd gradings (congruent to $\rho+ 1 \textnormal{ mod } 2$), the Pin(2)-homology is composed of the single $V$-tower $\mathcal{V}_b^+$. In even gradings (congruent to $\rho \textnormal{ mod } 2$), the rank of the Pin(2)-homology is half of the rank of the usual monopole Floer homology, rounded up.
\end{thm}
\begin{proof}
In sufficiently high gradings, the sublevel set $S_n$ is contractible by Corollary 3.2.5 of \cite{Nem3}. This implies that in each of these gradings the lattice cohomology consists of a single nonzero element lying in $F$. Let us consider the action of $U$ on such an element, which we assume to be in grading $2n$. From the geometric picture, we see that multiplication by $U^i$ takes this element to the sum of all the connected components in grading $2(n - i)$. Each $V$-tower may be identified by selecting an element of sufficiently high grading and repeatedly multiplying by $V$. From the fact that the $[r+3]$-submodule is identically zero, we see that the $\mathcal{V}_a^+$-tower must lie in the $[r]$-submodule (this will also be evident from the proof of Theorem \ref{prop6} itself). The above description of the $U$-action then identifies the $\mathcal{V}_a^+$-tower with the $[r]$-submodule of $\mathbb{H}'(\Gamma, k) \subseteq \mathbb{H}^0(\Gamma, k)/\im (1 + J)$, establishing the first equality. The description of the $Q$-action in Theorem \ref{prop6} shows that multiplication by $Q$ is an isomorphism from the $b$-tower to the $a$-tower, proving the second equality. Finally, the third equality follows from the remark preceding the theorem. The statement about the ranks of the Pin(2)-homology follows immediately from the fact that $\im (1 + J)$ has half the dimension of $E$.
\end{proof}
\noindent
Note that the invariants $\rho$ and $\delta$ can be read off from the ranks of the monopole Floer homology, recovering Theorem \ref{prop2}. 
\end{subsection}

\begin{subsection}{Pin(2)-Homology and the Gysin Sequence}
As mentioned at the beginning of the section, our computation hinges on the simple observation that the action of $J$ on $\mathbb{H}^0(\Gamma, k)$ may be identified with the composition of maps $\iota_* \circ \pi_*$ in the Gysin sequence (\ref{eq1}). This follows easily from the naturality of the isomorphism between monopole Floer homology and lattice cohomology, but since there are several maps that mediate this isomorphism, we sketch the details in the following lemma.

\begin{lem}\label{prop8}
Let $(Y, s)$ and $(\Gamma, k)$ be as above. The composition of maps $\iota_* \circ \pi_*$ in the Gysin sequence for $(- Y, s)$ coincides with the action of $1 + J$ on the lattice cohomology $\mathbb{H}^0(\Gamma, k)$ under the isomorphism of Theorem \ref{prop3}.
\end{lem}
\begin{proof}
Denote by $j$ the involution on the chain groups of the usual monopole Floer homology of $(-Y, s)$ coming from the $\mathbb{Z}/2\mathbb{Z}$-symmetry of the Chern-Simons-Dirac functional. We claim that the induced action of $j$ on the monopole homology coincides with the action of $J$ on $\mathbb{H}^0(\Gamma, k)$ under the isomorphism of Theorem \ref{prop3}. For this, we briefly recall the original formulation of $\mathbb{H}^0(\Gamma, k)$; see e.g.\ \cite{OS}. Given any element $x$ in the monopole Floer homology of $(-Y, s)$, let $\phi_x$ be the map from the lattice $[k] = k + 2L_\Gamma$ to $\mathcal{U}_0^+$ defined by 
\[
\phi_x(k') = \HM(W(\Gamma),k')(x)
\]
for any $k' \in [k]$. Here, $\HM(W(\Gamma), k')$ is the usual monopole Floer homology map
\[
\HM(W(\Gamma),k'): \HM(-Y, s) \rightarrow \HM(S^3)
\]
associated to the spin$^c$ structure $k'$ on the cobordism $W(\Gamma)$ between $Y$ and $S^3$, and we have identified the monopole Floer homology of $S^3$ with $\mathcal{U}_0^+$ (see e.g.\ \cite{KM}). According to Proposition 2.4 of \cite{OS}, the map $\phi_x$ satisfies a certain set of ``adjunction equalities" relating $\phi_x(k')$ and $\phi_x(k' + 2v)$ for all $k' \in [k]$ and $v \in L_\Gamma$. We temporarily denote the $\mathbb{F}[U]$-module formed by the class of all such maps from $[k]$ to $\mathcal{U}_0^+$ satisfying these equalities by $\mathbb{H}^0(\Gamma, k)$; this is the original definition of lattice cohomology and our present notion is simply a combinatorial re-formulation of it (see Theorem 3.1.12 of \cite{Nem3}). The main result of \cite{OS} is then that (for the 3-manifolds at hand) the correspondence $x \mapsto \phi_x$ is an $\mathbb{F}[U]$-module isomorphism from from the monopole Floer homology of $(-Y, s)$ onto the lattice cohomology $\mathbb{H}^0(\Gamma, k)$. \\
\\
The cobordism map $\HM(W(\Gamma), k')$ is defined by studying the moduli space of solutions to the four-dimensional Seiberg-Witten equations over the pair $(W(\Gamma), k')$ with some specified limiting behavior at the ends of the cobordism (see e.g.\ \cite{KM}). As in the three-dimensional case, there is a $j$-symmetry of the four-dimensional Seiberg-Witten equations, which (roughly speaking) restricts to the three-dimensional $j$-symmetry at the ends of the cobordism. However, since $k'$ need not be self-conjugate, this symmetry merely identifies the two moduli spaces of the four-dimensional equations corresponding to $(W(\Gamma), k')$ and $(W(\Gamma), -k')$. Points in the former limit to $x$ at one end, while points in the latter limit to $jx$. It follows that we have the equality
\[
\phi_{x}(k') = \phi_{jx}(-k')
\]
for all $x \in \HM(-Y, s)$ and $k' \in [k]$. Thus the correspondence $x \mapsto \phi_x$ takes the involution $j$ on the monopole Floer homology to the map on $\mathbb{H}^0(\Gamma, k)$ given by precomposing with the reflection $k' \mapsto -k'$. One then checks that under the appropriate combinatorial re-formulation, this latter involution coincides with the map $J$ defined at the beginning of the section. \\
\\
We now claim that the map on $\HM(-Y, s)$ given by $1 + j$ is equal to the composition of maps $\iota_* \circ \pi_*$ in the Gysin sequence (\ref{eq1}). To see this, we observe that the Gysin sequence is the homology exact sequence associated to the short exact sequence of chain complexes
\[
0 \xrightarrow{} C_*^{inv} \xrightarrow{i} C_* \xrightarrow{1+j} (1+j)C_* \xrightarrow{} 0.
\]
(See Proposition 3.10 of \cite{Lin1}.) Here, $C_*$ is the chain complex for the usual monopole Floer homology of $(-Y, s)$ and $C_*^{inv}$ is the subcomplex of $j$-invariant chains. The map $i$ is given by inclusion, and the subcomplex $(1+j)C_*$ is easily shown to be quasi-isomorphic to $C_*^{inv}$ via the inclusion of $(1+j)C_*$ in $C_*^{inv}$. Thus $\iota_* \circ \pi_*$ is evidently induced by the map $1 + j$ on $C_*$. There is a slight subtlety that arises from the fact that in the above sequence, the perturbation of the Chern-Simons-Dirac functional is taken to be $j$-equivariant, whereas in the standard definition of monopole Floer homology the perturbation is such that critical points are isolated. By Corollary 3.7 of \cite{Lin1}, however, the monopole Floer homology defined using such ``Morse-Bott" perturbations is canonically isomorphic to the usual monopole Floer homology, and it is easily checked that this isomorphism is $j$-equivariant.
\end{proof}
\noindent
We now proceed with the proof of Theorem \ref{prop6}.  
\begin{proof}
Let $q$ be a fixed grading of the monopole Floer homology of $(-Y, s)$. We begin by determining the Pin(2)-homology in grading $q$ when $q = \rho$ modulo $2$. Consider the splitting of the lattice cohomology in grading $q - \sigma$ given by 
\[
\mathbb{H}_{q - \sigma}^0(\Gamma, k) = \spa \{E_i\}_{i =1}^n \oplus \spa \{(1 + J)E_i\}_{i =1}^n \oplus F,
\]
where $F$ is at most one-dimensional. Under the isomorphism of Theorem \ref{prop3}, the composition $\iota_* \circ \pi_*$ maps the first summand isomorphically onto the second. Let $E' = \pi_* (\spa \{E_i\})$. Then $E$ and $E'$ fit into the subcomplex
\begin{equation}\label{eq2}
0 \rightarrow E' \xrightarrow{\iota_*} E = \spa \{(1 + J)E_i\} \oplus \spa \{E_i\} \xrightarrow{\pi_*} E' \rightarrow 0,
\end{equation}
where $\spa \{E_i\}$ maps isomorphically onto $E'$ via $\pi_*$ and $E'$ maps isomorphically onto the subspace $\spa \{(1 + J)E_i\}$ via $\iota_*$. Note that this subcomplex is simply $n$ copies of the exact sequence
\begin{equation}\label{eq5}
0 \rightarrow \mathbb{F} \rightarrow \mathbb{F} \oplus \mathbb{F} \rightarrow \mathbb{F} \rightarrow 0.
\end{equation}
\noindent
Roughly speaking, the intuition behind this decomposition is that the elements of $E$ consist of irreducible critical points which occur in pairs due to the $j$-symmetry of the Chern-Simons-Dirac functional. In the usual monopole Floer homology, each of these pairs contributes an $(\mathbb{F} \oplus \mathbb{F})$-summand, while in the Pin(2)-homology, each pair contributes a single $\mathbb{F}$-summand. Note that although the choice of $\{E_i\}$ is not canonical, the subspace $E'$ is canonically defined since $\pi_*(E) = \pi_*(\spa \{E_i\})$. \\
\\
Now let us consider $F$, which we assume for the moment to be nonzero in grading $q$. The fact that $(\iota_* \circ \pi_*)F = 0$ does not immediately allow us to determine how $F$ fits into the Gysin sequence. Instead, there are two cases. First, suppose that $\pi_* F \neq 0$. Since $(\iota_* \circ \pi_*)F = 0$, this implies that $\pi_* F$ is in the image of $Q$. Denote $\pi_* F$ by $F'$. Since the monopole Floer homology of $(-Y, s)$ is zero in grading $q + 1$, a simple diagram chase shows that $F$ and $F'$ lie in some subcomplex
\begin{equation}\label{eq3}
\begin{tikzpicture}[baseline=(current bounding box.center), scale=1.5]
\node (00) at (0,0) {$F'$};
\node (10) at (1,0) {$F$};
\node (20) at (2,0) {$F'$};
\node (01) at (0,1) {$\mathbb{F}$};
\node (11) at (1,1) {$0$};
\node (21) at (2,1) {$\mathbb{F}$};
\node (02) at (0,2) {$\mathbb{F}$};
\node (12) at (1,2) {$*$};
\node (22) at (2,2) {$\mathbb{F}$};

\path[->,font=\scriptsize,>=angle 90]
(10) edge node[below]{$\pi_*$} (20)
(21) edge node[above]{$\cdot Q$} (00)
(22) edge node[above]{$\cdot Q$} (01);
\end{tikzpicture}
\end{equation}
\noindent
On the other hand, if $\pi_* F$ is zero, then $\iota_*$ surjects onto $F$, and choosing any preimage $F'$ of $F$ under $\iota_*$, a similar diagram chase then shows that we have the subcomplex
\begin{equation}\label{eq4}
\begin{tikzpicture}[baseline=(current bounding box.center), scale=1.5]
\node (00) at (0,0) {$\mathbb{F}$};
\node (10) at (1,0) {$*$};
\node (20) at (2,0) {$\mathbb{F}$};
\node (01) at (0,1) {$\mathbb{F}$};
\node (11) at (1,1) {$0$};
\node (21) at (2,1) {$\mathbb{F}$};
\node (02) at (0,2) {$F'$};
\node (12) at (1,2) {$F$};
\node (22) at (2,2) {$F'$};

\path[->,font=\scriptsize,>=angle 90]
(21) edge node[above]{$\cdot Q$} (00)
(22) edge node[above]{$\cdot Q$} (01)
(02) edge node[above]{$\iota_*$} (12);
\end{tikzpicture}
\end{equation}
Here, we have used the fact that $F$ does not lie in the image of $1 +J$ to conclude that $F'$ cannot be in the image of $\pi_*$, and thus that the action of $Q$ on $F'$ is nonzero. Note that in this case the subspace $F'$ is not canonically determined but must be chosen as a preimage of $F$. \\
\\
We now claim that the Pin(2)-homology in grading $q$ is precisely equal to $E' \oplus F'$ in either of the two cases described above. Indeed, suppose we had an element $x$ of the Pin(2)-homology in grading $q$ lying outside of $E' \oplus F'$. We claim that without loss of generality we may assume $\iota_*x = 0$. Indeed, suppose that $\iota_*x \neq 0$. Since $(\iota_* \circ \pi_*)(\iota_* x) = 0$, we see that $\iota_*x$ must lie in $\ker (1 + J) = \spa \{(1 + J)E_i\} \oplus F$. By subtracting off elements of $E'$ from $x$, we may thus assume that $\iota_*x$ lies in $F$. If $\iota_*x$ is still nonzero, then $\pi_* F$ must be zero, and we are in the second case above where $F = \iota_* F'$. Subtracting off the nonzero element of $F'$ from $x$, we obtain an element lying outside of $E' \oplus F'$ such that $\iota_* x = 0$. \\
\\
Since the image of $\pi_*$ certainly lies in $E' \oplus F'$, we know that $x$ is not in the image of $\pi_*$. Thus $Qx \neq 0$. Putting everything together, we thus have that $x$ lies in the subcomplex
\[
\begin{tikzpicture}[scale=1.5]
\node (00) at (0,0) {$\mathbb{F}$};
\node (10) at (1,0) {$0$};
\node (20) at (2,0) {$\mathbb{F}$};
\node (01) at (0,1) {$x$};
\node (11) at (1,1) {$*$};
\node (21) at (2,1) {$x$};
\node (02) at (0,2) {$\mathbb{F}$};
\node (12) at (1,2) {$0$};
\node (22) at (2,2) {$\mathbb{F}$};

\path[->,font=\scriptsize,>=angle 90]
(21) edge node[above]{$\cdot Q$} (00)
(22) edge node[above]{$\cdot Q$} (01);
\end{tikzpicture}
\]
But this contradicts the fact that $Q^3 = 0$. Hence the Gysin sequence in grading $q$ is the direct sum of (\ref{eq2}) and either (\ref{eq3}) or (\ref{eq4}). A similar diagram chase, together with the fact that the monopole Floer homology of $(-Y, s)$ is supported only in even dimensions, shows that there can be no other elements of the Pin(2)-homology in gradings $q + 1$ or $q - 1$. Lemma \ref{prop4} thus implies that the entire Gysin sequence must be isomorphic to the direct sum of copies of (\ref{eq5}) (all in even gradings) and a tower of repeated copies of 
\[
\begin{tikzpicture}[scale=1.5]
\node (00) at (0,0) {$\mathbb{F}$};
\node (10) at (1,0) {$\mathbb{F}$};
\node (20) at (2,0) {$\mathbb{F}$};
\node (01) at (0,1) {$\mathbb{F}$};
\node (11) at (1,1) {$0$};
\node (21) at (2,1) {$\mathbb{F}$};
\node (02) at (0,2) {$\mathbb{F}$};
\node (12) at (1,2) {$\mathbb{F}$};
\node (22) at (2,2) {$\mathbb{F}$};

\path[->,font=\scriptsize,>=angle 90]
(10) edge node[above]{$\pi_*$} (20)
(02) edge node[above]{$\iota_*$} (12)
(21) edge node[above]{$\cdot Q$} (00)
(22) edge node[above]{$\cdot Q$} (01);
\end{tikzpicture}
\]
\noindent
stacked on top of each other with a grading shift of four. The lowest line of this tower has grading $r$ (or rather, $\rho$ in the monopole Floer homology). \\
\\
The above computation determines the Pin(2)-homology as an abelian group, and also specifies the $Q$-action. In order to determine the $V$-action, we must be a bit more circumspect as to the precise nature of the maps $\iota_*$ and $\pi_*$. Consider the four $\mathbb{F}[V]$-submodules in the statement of Theorem \ref{prop6}. In gradings $q = \rho + 3$ mod $4$, we see that the Pin(2)-homology is indeed identically zero. In gradings $q = \rho + 2$ mod $4$, the map $\iota_*$ is an $\mathbb{F}[V]$-module isomorphism from the Pin(2)-homology onto the subspace $\ker (1 + J) = \spa \{(1 + J)E_i\} \oplus F$, which establishes the second equality claimed in Theorem \ref{prop6}. In gradings $q = \rho$ mod $4$, the map $\pi_*$ surjects onto the Pin(2)-homology with kernel $\im (1 + J) = \spa \{(1 + J)E_i\}$, proving the fourth equality. \\
\\
It remains to establish the third equality and express the $Q$-action in terms of the claimed isomorphisms. From the proof of Theorem \ref{prop7}, we know that the $a$-tower may be identified with the appropriate submodule of $\mathbb{H}'(\Gamma, k) \subseteq \mathbb{H}^0(\Gamma, k)/\im (1 + J)$ in gradings $q -\sigma = r$ mod $4$. The above decomposition of the Gysin sequence then shows that multiplication by $Q$ is an isomorphism from the $[r + 1]$-submodule onto this tower; hence the $[r + 1]$-submodule is a single $V$-tower which we may also identify with $\mathbb{H}'(\Gamma, k)$. This proves the third equality and gives the $Q$-action from the $[r + 1]$- to the $[r]$-submodule. Finally, the $Q$-action from the $[r+2]$- to the $[r+1]$-submodule is injective on $F' \cong F \subseteq \ker (1 +J )$ and zero otherwise, which is the description of the $Q$-action given in Theorem \ref{prop6}.
\end{proof}
\end{subsection}

\begin{subsection}{Examples}
We close this section with a few basic examples. These are not new computations, but serve to illustrate the framework that we have established. As a visual aid, we use graded root diagrams to describe the lattice cohomology (see e.g.\ \cite{Nem2}).

\begin{exmp}[$\Sigma(2, 3, 5)$]\label{ex1}
The monopole Floer homology of $-\Sigma(2, 3, 5)$ is given by a single $U$-tower, $\mathcal{U}_{-2}^+$. (See e.g.\ Section 3.2 of \cite{OS}.) Accordingly, $\rho = 2\delta = -2$, and applying Theorem \ref{prop7}, we have that $a = -2$, $b = -1$, and $c = 0$. Thus the Pin(2)-homology is simply 
\[
\HS(-\Sigma(2, 3, 5)) = \mathcal{V}_0^+ \oplus \mathcal{V}_{-1}^+ \oplus \mathcal{V}_{-2}^+.
\]
Compare with \cite{Lin2}.
\end{exmp}

\begin{exmp}[$\Sigma(3, 5, 7)$]\label{ex2}
The monopole Floer homology of $-\Sigma(3, 5, 7)$ is given by 
\[
\HM(-\Sigma(3, 5, 7)) = \mathcal{U}_{-2}^+ \oplus \mathbb{F}_{(-2)} \oplus \mathbb{F}_{(0)} \oplus \mathbb{F}_{(0)},
\]
with the subscripts on each $\mathbb{F}$ indicating the grading. (See e.g.\ Section 3.2 of \cite{OS}.) With a slight change-of-basis, this is represented pictorially by the graded root in Figure \ref{fig1}.

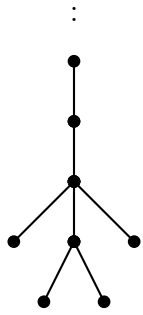
\begin{figure}[h!]
\begin{tikzpicture}[thick,scale=0.8]%
	\node[text width=0.1cm] at (0, 4) {\vdots};
	\draw (0, 0) node[circle, draw, fill=black!100, inner sep=0pt, minimum width=4pt] {} -- (0, 1) node[circle, draw, fill=black!100, inner sep=0pt, minimum width=4pt] {};
	\draw (0, 1) node[circle, draw, fill=black!100, inner sep=0pt, minimum width=4pt] {} -- (0, 2) node[circle, draw, fill=black!100, inner sep=0pt, minimum width=4pt] {};
	\draw (0, 2) node[circle, draw, fill=black!100, inner sep=0pt, minimum width=4pt] {} -- (0, 3) node[circle, draw, fill=black!100, inner sep=0pt, minimum width=4pt] {};
	
	\draw (0, 0) node[circle, draw, fill=black!100, inner sep=0pt, minimum width=4pt] {} -- (-0.5, -1) node[circle, draw, fill=black!100, inner sep=0pt, minimum width=4pt] {};
	\draw (0, 0) node[circle, draw, fill=black!100, inner sep=0pt, minimum width=4pt] {} -- (0.5, -1) node[circle, draw, fill=black!100, inner sep=0pt, minimum width=4pt] {};
	
	\draw (1, 0) node[circle, draw, fill=black!100, inner sep=0pt, minimum width=4pt] {} -- (0, 1) node[circle, draw, fill=black!100, inner sep=0pt, minimum width=4pt] {};
	\draw (-1, 0) node[circle, draw, fill=black!100, inner sep=0pt, minimum width=4pt] {} -- (0, 1) node[circle, draw, fill=black!100, inner sep=0pt, minimum width=4pt] {};
\end{tikzpicture}
\caption{Root diagram for $-\Sigma(3, 5, 7)$.}\label{fig1}
\end{figure}

\noindent
Here, each node represents an $\mathbb{F}$-summand, and edges correspond to multiplication by $U$. Note that our basis is chosen in such a way so that the $J$-action corresponds to reflection about the vertical axis in the diagram. Clearly, $\rho = 0$ and $2\delta = -2$. Applying Theorem \ref{prop7}, we have that $a = 0$, $b = 1$, and $c = - 2$. The Pin(2)-homology is given by 
\[
\HS(-\Sigma(3, 5, 7)) = (\mathcal{V}_{-2}^+ \oplus \mathcal{V}_1^+ \oplus \mathcal{V}_0^+) \oplus \mathbb{F}_{(0)}.
\]
\end{exmp}

\begin{exmp}[$\Sigma(2, 7, 15)$]\label{ex3}
The monopole Floer homology of $-\Sigma(2, 7, 15)$ is given by 
\[
\HM(-\Sigma(2, 7, 15)) = \mathcal{U}_{0}^+ \oplus (\mathbb{F}_{(0)} \oplus \mathbb{F}_{(2)}) \oplus \mathbb{F}_{(2)} \oplus \mathbb{F}_{(2)} \oplus \mathbb{F}_{(6)} \oplus \mathbb{F}_{(6)}
\]
with the action of $U$ taking the $\mathbb{F}_{(2)}$ inside of the parentheses onto the $\mathbb{F}_{(0)}$. (See e.g.\ \cite{Twee}.) After a change-of-basis, the corresponding graded root is given in Figure \ref{fig2}. 

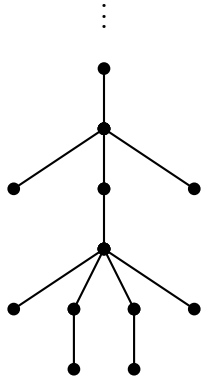
\begin{figure}[h!]
\begin{tikzpicture}[thick,scale=0.8]%
	\node[text width=0.1cm] at (0, 4) {\vdots};
	\draw (0, 0) node[circle, draw, fill=black!100, inner sep=0pt, minimum width=4pt] {} -- (0, 1) node[circle, draw, fill=black!100, inner sep=0pt, minimum width=4pt] {};
	\draw (0, 1) node[circle, draw, fill=black!100, inner sep=0pt, minimum width=4pt] {} -- (0, 2) node[circle, draw, fill=black!100, inner sep=0pt, minimum width=4pt] {};
	\draw (0, 2) node[circle, draw, fill=black!100, inner sep=0pt, minimum width=4pt] {} -- (0, 3) node[circle, draw, fill=black!100, inner sep=0pt, minimum width=4pt] {};
	
	\draw (0, 0) node[circle, draw, fill=black!100, inner sep=0pt, minimum width=4pt] {} -- (-0.5, -1) node[circle, draw, fill=black!100, inner sep=0pt, minimum width=4pt] {};
	\draw (0, 0) node[circle, draw, fill=black!100, inner sep=0pt, minimum width=4pt] {} -- (0.5, -1) node[circle, draw, fill=black!100, inner sep=0pt, minimum width=4pt] {};
	
	\draw (-0.5, -1) node[circle, draw, fill=black!100, inner sep=0pt, minimum width=4pt] {} -- (-0.5, -2) node[circle, draw, fill=black!100, inner sep=0pt, minimum width=4pt] {};
	\draw (0.5, -1) node[circle, draw, fill=black!100, inner sep=0pt, minimum width=4pt] {} -- (0.5, -2) node[circle, draw, fill=black!100, inner sep=0pt, minimum width=4pt] {};	
	
	\draw (0, 0) node[circle, draw, fill=black!100, inner sep=0pt, minimum width=4pt] {} -- (-1.5, -1) node[circle, draw, fill=black!100, inner sep=0pt, minimum width=4pt] {};
	\draw (0, 0) node[circle, draw, fill=black!100, inner sep=0pt, minimum width=4pt] {} -- (1.5, -1) node[circle, draw, fill=black!100, inner sep=0pt, minimum width=4pt] {};
	
	\draw (0, 2) node[circle, draw, fill=black!100, inner sep=0pt, minimum width=4pt] {} -- (-1.5, 1) node[circle, draw, fill=black!100, inner sep=0pt, minimum width=4pt] {};
	\draw (0, 2) node[circle, draw, fill=black!100, inner sep=0pt, minimum width=4pt] {} -- (1.5, 1) node[circle, draw, fill=black!100, inner sep=0pt, minimum width=4pt] {};
\end{tikzpicture}
\caption{Root diagram for $-\Sigma(2, 7, 15)$.}\label{fig2}
\end{figure}
\noindent
In this case, $\rho = 4$ and $2\delta = 0$. Applying Theorem \ref{prop7}, we have $a = 4$, $b = 5$, and $c = 2$. The Pin(2)-homology is given by 
\[
\HS(-\Sigma(2, 7, 15)) = (\mathcal{V}_{2}^+ \oplus \mathcal{V}_5^+ \oplus \mathcal{V}_4^+) \oplus \mathbb{F}_{(0)} \oplus \mathbb{F}_{(2)} \oplus \mathbb{F}_{(6)}.
\]
Note that in this example, the $a$-tower does not extend ``all the way down" - it stops at grading $a = 4$, even though there is a summand $\mathbb{F}_{(0)}$ lower than it in the same mod $4$ grading.
\end{exmp}

\noindent
The results of Theorem \ref{prop6} are easily converted into the language of graded roots by choosing a basis for the lattice cohomology in which the root diagram is symmetric about reflection through the vertical axis. (We remark that Lemma \ref{prop4} has an easy re-formulation by noticing that $r$ is precisely the grading at which the infinite ``central stem" of such a root either forks or vanishes.) Because the prescription of Theorem \ref{prop6} is straightforward, the Pin(2)-homology is easily computed once the lattice cohomology is known. See e.g.\ \cite{OS} or \cite{Nem1} for algorithms computing the lattice cohomology and \cite{Twee} for an extensive list of lattice cohomology calculations. We refer the reader to Lin \cite{Lin2} and Stoffregen \cite{Stoff} for more computations of the Pin(2)-homology of Seifert fibered spaces. 
\end{subsection}
\end{section}

\begin{section}{Applications and Further Developments}\label{sec3}
\noindent
In this section, we prove that Manolescu's conjecture $\beta(-Y, s) = \bar{\mu}(Y, s)$ holds for all plumbed 3-manifolds with at most one bad vertex. Since $\beta(-Y, s)$ coincides with half of the parity invariant $\rho$ by Theorem \ref{prop7}, this reduces to showing that half of $\rho$ coincides with $\bar{\mu}$. The fact that these are equal provides an interesting new interpretation of $\bar{\mu}$ in terms of the structure of the Heegaard Floer/monopole Floer homology of $(-Y, s)$. 

\begin{subsection}{Relations with the Neumann-Siebenmann Invariant}
We begin by recalling the definition of the Neumann-Siebenmann invariant (see \cite{Neu}, \cite{Sieb}). Let $Y$ be a plumbed 3-manifold with plumbing diagram $\Gamma$, and let $s$ be a spin structure on $Y$. Among the characteristic vectors $k$ on $\Gamma$ corresponding to spin$^c$ structures on $W(\Gamma)$ limiting to $s$ on $Y$, there is a unique vector $w$ such that all of the coordinates of $w$ in the natural basis of $L_\Gamma$ are either zero or one. (Note that since $s$ is self-conjugate, $w$ necessarily lies in $L_\Gamma$, rather than $L_\Gamma'$.) We then define:
\[
\bar{\mu}(Y, s) = \dfrac{1}{8}\left( \text{sign}(\Gamma) - w^2 \right).
\]
Here, $\text{sign}(\Gamma)$ is the signature of the intersection matrix of $\Gamma$ and $w^2$ is the self-pairing of $w$. It is shown in \cite{Neu} that $\bar{\mu}$ is an integer lift of the Rohklin invariant defined for plumbed rational homology spheres, and is independent of the choice of plumbing. The set of basis vectors having nonzero coefficient in $w$ is referred to as the \textit{Wu set} of the pair $(Y, s)$. Our proof that $\rho(Y, s) = 2\bar{\mu}(Y, s)$ rests on the well-known fact that any Wu set consists of pairwise non-adjacent vertices in $\Gamma$. For completeness, we give a proof of this lemma below: 

\begin{lem}\label{lemconj}
Let $\Gamma$ be any plumbing tree, and let $w$ be a characteristic vector for $\Gamma$ such that all of the coordinates of $w$ in the natural basis of $L_\Gamma$ are either zero or one. After appropriate permutation, we may assume that $w = e_1 + e_2 + \cdots + e_n$ for some $0 \leq n \leq |\Gamma|$. Then $(e_i, e_j) = 0$ for all $1 \leq i \neq j \leq n$; that is, the vertices $e_i$ for $1 \leq i \leq n$ are pairwise non-adjacent in $\Gamma$.
\end{lem}
\begin{proof}
Let $\Gamma'$ be the induced subgraph of $\Gamma$ spanned by the vertices corresponding to $e_1, e_2, \ldots, e_n$. Because $w$ is characteristic, we have that
\[
(e_1 + e_2 + \cdots e_n, e_i) = (e_i, e_i) \mod 2
\]
for all $1 \leq i \leq n$. Since $\Gamma$ is a plumbing tree, two distinct vertices have pairing one precisely when connected by an edge in $\Gamma$, and have pairing zero otherwise. Hence the above equality shows that in our induced subgraph $\Gamma'$, every vertex has an even number of adjacencies. But $\Gamma$ (and thus $\Gamma'$) has no cycles, so the only way for this to be possible is for every vertex to be isolated. This proves the lemma.
\end{proof}
\noindent
We now turn to the computation of $\rho(Y, s)$:
\begin{thm}\label{propmu}
Let $Y$ be a rational homology 3-sphere given by surgery on a connected, negative-definite graph with at most one bad vertex (in the sense of \cite{OS}). Let $s$ be a spin structure on $Y$ (which we may view as a self-conjugate spin$^c$ structure). Then $\rho(Y, s) = 2\bar{\mu}(Y, s)$.
\end{thm}
\begin{proof}
Let $[k]$ be the equivalence class of characteristic vectors on $\Gamma$ corresponding to spin$^c$ structures on $W(\Gamma)$ limiting to $s$ on $Y$. Let $w \in [k]$ be the representative described in the definition of the Neumann-Siebenmann invariant, so that $\bar{\mu}(Y, s) = (-|\Gamma | - w^2)/8$. When computing the lattice cohomology of $\Gamma$, we are free to choose any representative of $[k]$, but the most convenient choice is obviously $w$. Then the grading shift of Theorem \ref{prop3} is given by $\sigma(\Gamma, w) = (- |\Gamma| - w^2)/4$. Since this is already twice the Neumann-Siebenmann invariant, we must thus show that the parity invariant $r$ of $\mathbb{H}^0(\Gamma, w)$ is zero. Now, the rank of $\mathbb{H}^0(\Gamma, w)$ is odd in grading $2n$ precisely when the sublevel set $S_n$ has a connected component which is invariant under the action of $J$; as shown in the proof of Lemma \ref{prop4}, such a connected component is uniquely characterized by containing the point of reflection $-w/2$. Hence it suffices to show that $-w/2 \in S_n$ precisely when $n \geq 0$. \\
\\
In order to illustrate the intuition behind the proof, suppose for a moment that $w$ is in fact zero (which is equivalent to the pairing of $\Gamma$ being even). Then the weight $w_0$ of $-w/2$ (as defined in Section \ref{sec1}) is zero, and this immediately implies the claim. If $w$ is not zero, however, then $-w/2$ does not lie in the lattice $L_\Gamma$, and in order to determine whether $-w/2$ lies in a sublevel set, we must instead compute the weight of the smallest lattice cube of $L_\Gamma$ containing $-w/2$. After an appropriate permutation of the natural basis of $L_\Gamma$, we may assume that
\[
w = e_1 + e_2 + \cdots + e_n
\]
for some $0 \leq n \leq |\Gamma|$. Let $C_n$ be the vertices of the $n$-dimensional lattice cube containing the sides $e_1, e_2, \ldots, e_n$; that is, define
\[
C_n = \{c_1e_1 + c_2e_2 + \cdots c_ne_n : \text{each }c_i = 0 \text{ or }1\}.
\]
Then the smallest lattice cube containing $-w/2$ is $n$-dimensional and is given by the translate $-w + C_n$. Since the weight of a lattice cube is defined to be the maximum over the weights of its vertices, it suffices to show that $w_0(-w + v) = 0$ for all $v \in C_n$. Now, since
\[
w_0(-w + v) = -((-w + v, -w + v) + (-w + v, w))/2 = ((w, v) - (v, v))/2,
\]
this is equivalent to showing that $(w, v) = (v, v)$ for all $v \in C_n$. (Note that since $w$ is characteristic, this equality is always true modulo two, but it is strict equality that we must establish.) By Lemma \ref{lemconj}, however, we have that $(e_i, e_j) = 0$ whenever $1 \leq i \neq j \leq n$. Expanding the pairings $(w, v)$ and $(v, v)$ immediately establishes the equality and completes the proof.
\end{proof}
\noindent
For plumbed 3-manifolds with at most one bad vertex, Theorem \ref{propmu} provides an easy characterization of $\bar{\mu}(Y, s)$ in terms of the ranks of the Heegaard Floer/monopole Floer homology of $(-Y, s)$. If we also happen to know that $\rho$ coincides with (twice) the Fr\o yshov invariant of $(-Y, s)$, then we additionally obtain a relation between $\bar{\mu}$ and the Ozsv\'ath-Szab\'o $d$-invariant. This occurs, for example, if $Y$ is known to be an $L$-space, although certainly the condition of being an $L$-space is not necessary. In particular, for rational surface singularities and spherical 3-manifolds, we recover the results of Stipsicz \cite{Stip} and Ue \cite{Ue}. \\
\\
Finally, we observe that Theorem \ref{prop7} and Theorem \ref{propmu} together imply Theorem \ref{propconj}. This proves Manolescu's conjecture for all plumbed 3-manifolds with at most one bad vertex. 
\end{subsection}


\begin{subsection}{Further Developments}
\noindent
In this subsection, we prove some tentative results aimed at generalizing our computations to a larger class of manifolds. There are two main difficulties with attempting to extend Theorems \ref{prop6} and \ref{prop7}. First, the precise relation between lattice cohomology and Heegaard Floer/monopole Floer homology in the case of arbitrary negative-definite plumbings is currently unknown. In \cite{OSS} it is shown that there exists a spectral sequence from lattice homology to Heegaard Floer homology, but at the moment the two are only conjecturally isomorphic. (See the discussion preceding Example \ref{extwobadv} for the two-bad-vertex case.) Nevertheless, one can still ask what extra information is needed to determine the Pin(2)-homology from the Gysin sequence. For instance, in the situation of Section \ref{sec2}, we saw that knowledge of the map $\iota_* \circ \pi_*$ was sufficient, and we showed that this additional data could be found in the lattice complex. \\
\\
Unfortunately, the proof of this sufficiency relied on the fact that the monopole Floer homology was supported only in even gradings, a structure result that certainly does not hold in general. Our first result identifies some extra algebraic data that (in theory) suffices to determine the Pin(2)-homology (at least as an abelian group) once the monopole Floer homology is known. \\
\\
We begin by understanding the possible decompositions of the Gysin sequence. Denote by $\mathcal{I}_0$ the exact sequence
\[
\mathbb{F} \xrightarrow{\iota_*} \mathbb{F} \oplus \mathbb{F} \xrightarrow{\pi_*} \mathbb{F},
\]
and by $\mathcal{I}_1$ the exact sequence
\[
\begin{tikzpicture}[scale=1.5]
\node (00) at (0,0) {$\mathbb{F}$};
\node (10) at (1,0) {$\mathbb{F}$};
\node (20) at (2,0) {$\mathbb{F}$};
\node (01) at (0,1) {$\mathbb{F}$};
\node (11) at (1,1) {$\mathbb{F}$};
\node (21) at (2,1) {$\mathbb{F}$};

\path[->,font=\scriptsize,>=angle 90]
(10) edge node[above]{$\pi_*$} (20)
(01) edge node[above]{$\iota_*$} (11)
(21) edge node[above]{$\cdot Q$} (00);
\end{tikzpicture},
\]
and finally by $\mathcal{I}_2$ the exact sequence
\[
\begin{tikzpicture}[scale=1.5]
\node (00) at (0,0) {$\mathbb{F}$};
\node (10) at (1,0) {$\mathbb{F}$};
\node (20) at (2,0) {$\mathbb{F}$};
\node (01) at (0,1) {$\mathbb{F}$};
\node (11) at (1,1) {$0$};
\node (21) at (2,1) {$\mathbb{F}$};
\node (02) at (0,2) {$\mathbb{F}$};
\node (12) at (1,2) {$\mathbb{F}$};
\node (22) at (2,2) {$\mathbb{F}$};

\path[->,font=\scriptsize,>=angle 90]
(10) edge node[above]{$\pi_*$} (20)
(02) edge node[above]{$\iota_*$} (12)
(21) edge node[above]{$\cdot Q$} (00)
(22) edge node[above]{$\cdot Q$} (01);
\end{tikzpicture}.
\]
Here, every pair of $\mathbb{F}$-summands on the left and the right represent the same element in the Pin(2)-homology, and consecutive lines in each sequence differ by a grading shift of one. We let the lowest line of each sequence have grading zero, so that $\mathcal{I}_n[d]$ is the exact sequence $\mathcal{I}_n$ shifted so that the lowest line has grading $d$. \\
\\
Because $Q^3 = 0$ and we are working over $\mathbb{F}_2$, a straightforward diagram chase shows that each element in the Gysin sequence lies in a subcomplex isomorphic to one of the above three. Hence the Gysin sequence decomposes (non-canonically) into a direct sum of $\mathcal{I}_0$, $\mathcal{I}_1$, and $\mathcal{I}_2$. Setting aside the structure of the Pin(2)-homology as an $\mathcal{R}$-module, one might then ask whether the isomorphism class of the Gysin sequence (and thus the ranks of the Pin(2)-homology) can, in general, be determined from the lattice cohomology. \\
\\
In order to approach this question, we first recall that the Gysin sequence may be expressed in the language of an exact couple:
\[
\begin{tikzpicture}[scale=1.5]
\node (A) at (0,1.5) {$A = \HM(Y, s)$};
\node (B) at (-1.2,0) {$\HS(Y, s)$};
\node (C) at (1.2,0) {$\HS(Y, s)$};

\path[->,font=\scriptsize,>=angle 90]
(A) edge node[right]{$\pi_*$} (C)
(C) edge node[above]{$\cdot Q$} (B)
(B) edge node[left]{$\iota_*$} (A);
\end{tikzpicture}.
\]
Taking the derived couple of this results in another exact couple whose derived homology group (in the case of Lemma \ref{prop4}) is precisely the derived lattice cohomology under the isomorphism of Theorem \ref{prop3}. However, in our situation, the lattice cohomology is not necessarily isomorphic to the monopole Floer homology, so we denote the derived homology group simply by $A'$:
\[
\begin{tikzpicture}[scale=1.5]
\node (A) at (0,1.5) {$A' = H(A, \iota_* \circ \pi_*)$};
\node (B) at (-1.2,0) {$\im Q$};
\node (C) at (1.2,0) {$\im Q$};

\path[->,font=\scriptsize,>=angle 90]
(A) edge node[right]{$\pi_*'$} (C)
(C) edge node[above]{$\cdot Q$} (B)
(B) edge node[left]{$\iota_*'$} (A);
\end{tikzpicture}.
\]
Here, $A'$ is the homology of $A$ with respect to the differential $\iota_* \circ \pi_*$, the group $\im Q$ is the image of multiplication by $Q$, and the maps $\iota_*'$ and $\pi_*'$ are the usual induced maps in the derived couple. Deriving one more time results in a third exact couple, but because $Q^3 = 0$, this degenerates into the short exact sequence
\[
0 \rightarrow \im Q^2 \rightarrow A'' = H(A', \iota_*' \circ \pi_*') \rightarrow \im Q^2 \rightarrow 0.
\]
\noindent
\\
We now claim that the two groups $A'$ and $A''$ (along with the original monopole Floer homology) suffice to determine the isomorphism class of the Gysin sequence.
\begin{thm}\label{prop9}
Let the monopole Floer homology of $\HM(Y, s)$ be fixed. If, in addition, we know the ranks of the derived groups $A'$ and $A''$, then the Pin(2)-equivariant monopole Floer homology of $(Y, s)$ is determined as an abelian group.
\end{thm}
\begin{proof}
As above, the Gysin sequence for $\HM(Y, s)$ decomposes into a direct sum of copies of $\mathcal{I}_0$, $\mathcal{I}_1$, and $\mathcal{I}_2$. Let us see how these are related to the ranks of $A'$ and $A''$. First, observe that elements of the monopole Floer homology lying in a complex isomorphic to $\mathcal{I}_0$ do not survive to the subquotient $A'$, as either they are not in the kernel of $\iota_* \circ \pi_*$, or they are in the image of $\iota_* \circ \pi_*$. In contrast, elements of the monopole Floer homology lying in a $\mathcal{I}_1$- or $\mathcal{I}_2$-summand each contribute an $\mathbb{F}$-summand to $A'$. Unwinding the definitions of $\iota_*'$ and $\pi_*'$, a similar result holds for $A''$: the elements of $A'$ represented by elements coming from $\mathcal{I}_1$-summands do not live to the $A''$ subquotient, while those coming from the $\mathcal{I}_2$-summands each contribute an $\mathbb{F}$-summand to $A''$. More precisely, suppose the Gysin sequence is isomorphic to

\[
\left(\bigoplus_{i} \mathcal{I}_0[n_i] \right) \oplus \left( \bigoplus_{j} \mathcal{I}_1[n_j] \right) \oplus \left( \bigoplus_{k} \mathcal{I}_2[n_k] \right).
\]
\noindent
\\
Then the derived groups $A'$ and $A''$ are given by
\begin{equation}\label{eqnap}
A' = \left( \bigoplus_{j} \left((\mathbb{F}_{(n_j)} \oplus \mathbb{F}_{(n_j + 1)} \right) \right) \oplus \left( \bigoplus_{k} \left((\mathbb{F}_{(n_k)} \oplus \mathbb{F}_{(n_k + 2)} \right) \right)
\end{equation}
\noindent
\\
and
\begin{equation}\label{eqnapp}
A'' = \bigoplus_{k} \left((\mathbb{F}_{(n_k)} \oplus \mathbb{F}_{(n_k + 2)} \right).
\end{equation}
\noindent
\\
Now suppose that the ranks of $A''$ are known. Because the monopole Floer homology of $(Y, s)$ is zero in sufficiently low gradings, the same is true for $A''$. Hence the above argument can easily be inverted to determine the numbers $n_k$. Let $d$ be the grading below which $A''$ vanishes. Then for each $n$, the number of copies of $\mathcal{I}_2[n]$ in our decomposition of the Gysin sequence is given by the alternating sum
\[
\dim A''_n - \dim A''_{n-2} + \dim A''_{n-4} \cdots \pm \dim A''_d
\]
if $n = d$ mod $2$, and
\[
\dim A''_{n} - \dim A''_{n-2} + \dim A''_{n-4} \cdots \pm \dim A''_{d+1}
\]
otherwise. Once the placement of the $\mathcal{I}_2$-summands in the decomposition of the Gysin sequence is known, we can similarly determine the summands $\mathcal{I}_1[n_j]$ from the ranks of $A'$. Finally, once both sets of $\mathcal{I}_2$- and $\mathcal{I}_1$-summands are known, the ranks of the original monopole Floer homology determine the placement and number of the $\mathcal{I}_0[n_i]$.
\end{proof}
\noindent
Note that the composition $\iota_*' \circ \pi_*'$ has grading shift $-1$. Hence when the monopole Floer homology is only supported in even dimensions, this map is identically zero and the two groups $A'$ and $A''$ are equal. In this case there are no $\mathcal{I}_1$-summands, and the Gysin sequence is entirely determined by $A'$, exactly as in Theorem \ref{prop6}. \\

\noindent
We now specialize to the case of a plumbed 3-manifold with at most two bad vertices. The following sharpening of Theorem \ref{prop3} was essentially established in \cite{OSS}, with elements appearing previously in \cite{OS}, \cite{Nem3}:

\begin{thm}[Corollary 1.3 of \cite{OSS}]\label{prop99}
Let $Y$ be a rational homology 3-sphere obtained by surgery on a negative-definite graph $\Gamma$ with at most two bad vertices. Let $s$ be a spin$^c$ structure on $Y$, and let $k$ be any characteristic vector on $\Gamma$ whose corresponding spin$^c$ structure on $W(\Gamma)$ limits to $s$. Let $\sigma$ be the rational grading shift
\[
\sigma = \sigma(\Gamma, k) = - \dfrac{1}{4}(|\Gamma| + k^2),
\]
where $|\Gamma|$ is the number of vertices in $\Gamma$. Then the following are true:
\begin{enumerate}
\item $\mathbb{H}^q(\Gamma, k) = 0$ for all $q > 1$,
\item $\HM_{\text{even}}(-Y, s) \cong \mathbb{H}^0(\Gamma, k)[\sigma]$ as graded $\mathbb{F}[U]$-modules, and
\item $\HM_{\text{odd}}(-Y, s) \cong \mathbb{H}^1(\Gamma, k)[\sigma-1]$ as graded $\mathbb{F}[U]$-modules. 
\end{enumerate}
\end{thm}
\noindent
The second assertion in Theorem \ref{prop99} was actually shown in \cite{OS} and \cite{Nem3} using the same map between Heegaard Floer/monopole Floer homology and lattice cohomology as outlined in the proof of Lemma \ref{prop8}. In particular, the identification between $\iota_* \circ \pi_*$ and $1+J$ holds in all even gradings for the two-bad-vertex case. Unfortunately, the proof of the third part of Theorem \ref{prop99} is somewhat less direct and relies on the collapsing of a spectral sequence from lattice homology to Heegaard Floer homology constructed in \cite{OSS}. Thus it is not clear (although certainly a reasonable conjecture) that $\iota_* \circ \pi_*$ coincides with $1 + J$ in odd gradings also. \\
\\
We now give a computation of the Pin(2)-homology of a two-bad-vertex manifold in which the relative simplicity of the Floer homology allows us to determine the derived groups $A'$ and $A''$ algebraically.

\begin{exmp}[Example 4.4.1 of \cite{Nem3}]\label{extwobadv}
Let $\Gamma$ be the plumbing diagram given in Figure \ref{fig3}.

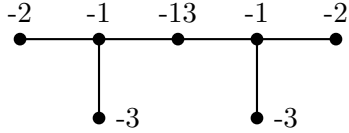
\begin{figure}[h!]
\begin{tikzpicture}[thick,scale=0.7]%

	\draw (0, 0) node[circle, draw, fill=black!100, inner sep=0pt, minimum width=4pt, label=-13] {};
	\draw (1.5, 0) node[circle, draw, fill=black!100, inner sep=0pt, minimum width=4pt, label=-1] {}; 
	\draw (-1.5, 0) node[circle, draw, fill=black!100, inner sep=0pt, minimum width=4pt, label=-1] {}; 
	\draw (3, 0) node[circle, draw, fill=black!100, inner sep=0pt, minimum width=4pt, label=-2] {};  
	\draw (-3, 0) node[circle, draw, fill=black!100, inner sep=0pt, minimum width=4pt, label=-2] {};
	\draw (1.5, -1.5) node[circle, draw, fill=black!100, inner sep=0pt, minimum width=4pt, label=right:-3] {}; 
	\draw (-1.5, -1.5) node[circle, draw, fill=black!100, inner sep=0pt, minimum width=4pt, label=right:-3] {}; 
	
	\path[-,font=\scriptsize,>=angle 90]
	(0, 0) edge (1.5, 0)
	(0, 0) edge (-1.5, 0)
	(1.5, 0) edge (3, 0)
	(-1.5, 0) edge (-3, 0)
	(1.5, 0) edge (1.5, -1.5)
	(-1.5, 0) edge (-1.5, -1.5);
\end{tikzpicture}
\caption{Plumbing diagram $\Gamma$ for Example \ref{extwobadv}.}\label{fig3}
\end{figure}

\noindent
One can check that $\Gamma$ is unimodular and thus that the plumbed manifold $Y = \partial W(\Gamma)$ is an integer homology sphere. In Example 4.4.1 of \cite{Nem3} a specific characteristic vector $k$ on $\Gamma$ is given for which the corresponding grading shift can be calculated to be $\sigma(\Gamma, k) = 2$. With respect to this $k$, the zeroth lattice cohomology is shown to be
\[
\mathbb{H}^0(\Gamma, k) = \mathcal{U}_{-2}^+ \oplus \mathbb{F}_{(-2)} \oplus \mathbb{F}_{(0)} \oplus \mathbb{F}_{(0)}.
\]
It is also established that the first lattice cohomology consists of a single generator in the sublevel set of weight zero. Taking into account the grading shift $\sigma$, the monopole Floer homology of $-Y$ is thus given by
\[
\HM(-Y) = \mathcal{U}_0^+ \oplus \mathbb{F}_{(0)} \oplus \mathbb{F}_{(1)} \oplus \mathbb{F}_{(2)} \oplus \mathbb{F}_{(2)}.
\]
We now apply Lemma \ref{prop8} to compute the derived group $A'$ in even gradings. One can verify that the $J$-action on the lattice cohomology is as follows. In the lowest (shifted) grading zero, the lattice cohomology consists of a symmetric pair of connected components which are taken to each other under the action of $J$. (In the decomposition above, the generator of $\mathcal{U}_0^+$ in grading zero is represented by the sum of these two components, while the summand $\mathbb{F}_{(0)}$ is represented by either one.) In (shifted) grading two, the lattice cohomology consists of three connected components, two of which occur in a symmetric pair and the third of which has nontrivial cohomology and is taken to itself by $J$. (Again, the generator of $\mathcal{U}_0^+$ in grading two corresponds to the sum of all three of these, while $\mathbb{F}_{(2)} \oplus \mathbb{F}_{(2)}$ is represented by the symmetric pair.) In all other even gradings $J$ is the identity. Applying Lemma \ref{prop8}, this shows that the even part of $A'$ is given by $\mathbb{F}_{(2)} \oplus \mathbb{F}_{(4)} \oplus \mathbb{F}_{(6)} \oplus \cdots$. Moreover, observe that the odd part of the Floer homology consists of a single generator in grading one. Hence the action of $\iota_* \circ \pi_*$ in odd gradings is either zero or the identity, and since $\iota_* \circ \pi_*$ squares to zero, it cannot be the identity. This shows that
\[
A' = \mathbb{F}_{(1)} \oplus \left(\mathbb{F}_{(2)} \oplus \mathbb{F}_{(4)} \oplus \mathbb{F}_{(6)} \oplus \cdots\right).
\]
With $A'$ in hand, we now wish to compute $A''$. Comparing our expression for $A'$ with the form of (\ref{eqnap}) and (\ref{eqnapp}), however, we see that the action of $\iota_*' \circ \pi_*'$ is already algebraically determined. Indeed, the only way for our computation of $A'$ to be consistent with the fact that $A''$ consists of pairs of generators separated by a grading difference of two is for $\iota_*' \circ \pi_*'$ to be an isomorphism from $\mathbb{F}_{(2)}$ to $\mathbb{F}_{(1)}$ and zero everywhere else. Applying Theorem \ref{prop9}, this shows that the Gysin sequence is given by
\[
\mathcal{I}_0[0] \oplus \mathcal{I}_1[1] \oplus \mathcal{I}_0[2] \oplus \left( \bigoplus_{n\geq0} \mathcal{I}_2[4 + 2n] \right).
\]
Moreover, it turns out that in this case the Gysin sequence determines the $\mathbb{F}[V]$-module structure. A similar argument as in Theorem \ref{prop6} shows that
\[
\HS(-Y) = (\mathcal{V}_2^+ \oplus \mathcal{V}_1^+ \oplus \mathcal{V}_4^+) \oplus \mathbb{F}_{(0)} \oplus \mathbb{F}_{(2)}.
\]
In particular, we have $\alpha = 2$ and $\beta = \gamma = 0$.
\end{exmp}

\noindent
We close with a conjecture on the computation of $A'$ and $A''$ in the general case. Let $Y$ be a plumbed 3-manifold with at most two bad vertices. As remarked previously, one obvious conjecture would be to identify $\iota_* \circ \pi_*$ with $1 + J$ in both even and odd gradings. Slightly more subtle is the question of finding an analogue of the map $\iota_*' \circ \pi_*'$ in lattice cohomology. For this, we proceed by finding a different Gysin sequence into which the lattice cohomology fits. \\
\\
Following Theorem \ref{prop99}, consider the graded $\mathbb{F}[U]$-module
\[
\mathbb{H}^{\textit{tot}}(\Gamma, k) = \mathbb{H}^0(\Gamma, k) \oplus \mathbb{H}^1(\Gamma, k)[-1].
\]
Each sublevel set $S_n$ in the lattice cohomology of $Y$ fits into the usual Gysin sequence of spaces relating the regular cohomology of $S_n$ with its Borel $\mathbb{Z}/2\mathbb{Z}$-equivariant cohomology. By summing the Gysin sequences of all these sublevel sets together (with the caveat that increasing degree in singular cohomology corresponds to decreasing grading in the total lattice complex), we obtain an exact sequence relating $\mathbb{H}^{\textit{tot}}(\Gamma, k)$ with the Borel equivariant cohomology of the lattice complex. Viewing this as an exact couple as before, we similarly obtain derived groups $B'$ and $B''$, the first of which is the derived lattice cohomology $\ker(1 + J)/ \im(1 + J)$. We then have:

\begin{conj}\label{conj}
Let $Y$ be a plumbed 3-manifold with at most two bad vertices and let $s$ be a self-conjugate spin$^c$ structure on $Y$. Then the derived groups $A'$ and $A''$ in the Gysin sequence (\ref{eq1}) for $(-Y, s)$ are isomorphic (up to a grading shift) to the derived groups $B'$ and $B''$ in the Gysin sequence of spaces for the lattice cohomology.
\end{conj}
\noindent
Note that the Pin(2)-homology is not isomorphic to the Borel equivariant cohomology of the lattice complex, even as an abelian group. Our claim is instead that the maps $\iota_* \circ \pi_*$ and $\iota_*' \circ \pi_*'$ in the two Gysin sequences coincide. It can be shown that for manifolds with at most two bad vertices, $B'$ and $B''$ indeed have the form of (\ref{eqnap}) and (\ref{eqnapp}) and that the higher derived groups vanish. \\
\\
Conjecture \ref{conj} is certainly true in Example \ref{extwobadv}, and can also be verified in a number of other examples which are similar in nature. We expect that some lattice-cohomology characterization of $\alpha$, $\beta$, and $\gamma$ as in Theorem \ref{prop7} holds for two-bad-vertex manifolds, and indicate this as a further area of research and possible application to the Pin(2)-homology of connected sums (see e.g.\ \cite{Stoff2}, \cite{Lin3}). (Indeed, several partial results along the lines of Example \ref{extwobadv} can already be obtained, but a full treatment of these lies outside the scope of this paper.)
\end{subsection}
\end{section}

\end{document}